\documentclass[12pt,a4paper,oneside]{article}
\usepackage{amsmath,amsfonts,amssymb,latexsym}

\newtheorem{theorem}{Theorem}[section]

\newtheorem{corollary}[theorem]{Corollary}
\newtheorem{remark}[theorem]{Remark}

\newenvironment{proof}
{\bigskip\noindent{\sc Proof.}\ \ \rm }{\hfill$\Box$\bigskip}
\title{ Some limit theorems for rescaled  Wick powers }
\author{ Alberto Lanconelli }
\date{\empty}
\begin{document}
\maketitle

\!\!\!\!\!\!\!\!\!\!\!\!\!\!\!\!\!\!\!\!
\begin{center}
{ \noindent
\begin{tabular}{cc}
& Dipartimento di Matematica\\
& Universita' degli Studi di Bari\\
& Via E. Orabona, 4\\
& 70125 Bari - Italia\\
& E-mail: lanconelli@dm.uniba.it \\

\end{tabular}
}
\end{center}

\numberwithin{equation}{section}

\bigskip

\begin{abstract}
We establish the strong $\mathcal{L}^2(\mathcal{P})$-convergence of
properly rescaled Wick powers as the power index tends to infinity.
The explicit representation of such limit will also provide the
convergence in distribution to normal and log-normal random
variables. The proofs rely on some estimates for the
$\mathcal{L}^2(\mathcal{P})$-norm of Wick products and on the
properties of second quantization operators.
\end{abstract}

     Key words and phrases: Wick product, second quantization operator, convergence in distribution.\\
AMS 2000 classification: 60H40, 60F25, 60H15.
\section{ Introduction }

In the last decade several authors have identified the Wick product
as a necessary tool for the study of certain types of stochastic
partial differential equations (SPDEs) or for the solution to some
related problems. This is motivated by the crucial features of the
Wick product: firstly it represents a bridge between stochastic and
classical integration theories; secondly it provides an efficient
way of multiplying infinite dimensional distributions. Since several
SPDEs of interest do not admit classical solutions, the possibility
of treating nonlinearities becomes fundamental. Important examples
in these regards are the stochastic quantization equation, which was
studied among others in \cite{DD},\cite{DT} and \cite{GG}, and the
KPZ equation, studied for instance in \cite{BG} and \cite{C}. Also
the problem of finding It\^o's type formulas for SPDEs leads in a
natural way to the use of Wick powers as a renormalization
technique. This is shown in \cite{Z} and \cite{L}. We also mention
the book \cite{HOUZ} which proposes a systematic use of Wick
products for the formulation and the study of a variety of SPDEs.\\

Let us briefly introduce the Wick product (see the next section for
precise definitions). Consider two multiple It\^o integrals
$I_n(h_n)$ and $I_m(g_m)$ where $h_n\in\mathcal{L}^2([0,T]^n)$ and
$g_m\in\mathcal{L}^2([0,T]^m)$. The well-known Hu-Meyer formula
establishes that
\begin{eqnarray}
I_n(h_n)\cdot I_m(g_m)=\sum_{r=0}^{n\wedge m}r!{n \choose r}{m
\choose r}I_{n+m-2r}(h_n\hat{\otimes}_rg_m),
\end{eqnarray}
where
\begin{eqnarray}
(h_n\otimes_rg_m)(t_1,...,t_{n+m-2r}):=\int_{[0,T]^r}h_n(t_1,...,t_{n-r},\bold{s})
g_m(t_{n-r+1},...,t_{n+m-2r},\bold{s})d\bold{s},
\end{eqnarray}
and $\hat{}$ stands for symmetrization. If $h_n$ and $g_m$ are only
tempered distributions then $I_n(h_n)$ and $I_m(g_m)$ become
generalized random variables. In this case the pointwise product
(1.1) is not anymore well defined due to the presence of the trace
terms (1.2) that do not make sense in this new situation. To
overcome this problem one can drop these problematic terms and
define  a new product obtained keeping only the term with $r=0$ in
the sum (1.1), namely:
\begin{eqnarray*}
I_n(h_n)\diamond I_m(g_m):=I_{n+m}(h_n\hat{\otimes} g_m).
\end{eqnarray*}
This is called \emph{Wick product} of $I_n(h_n)$ and $I_m(g_m)$. If
for example we denote by $W_t$ the time derivative of a Brownian
motion $B_t$ then we can write $W_t=I_1(\delta_t)$ ($\delta_t$
stands for the Dirac's delta function concentrated in $t$) and
obtain
\begin{eqnarray*}
W_t\diamond W_t=I_2(\delta_t\otimes\delta_t).
\end{eqnarray*}
Observe that applying formally (1.1) we get
\begin{eqnarray*}
W_t\diamond W_t="W_t^2-\int_0^T\delta_t^2(s)ds".
\end{eqnarray*}
The above mentioned bridge between stochastic and classical
integration theories can be now formalized precisely as
\begin{eqnarray*}
\int_0^T\xi_tdB_t=\int_0^T\xi_t\diamond W_t dt,
\end{eqnarray*}
where the left hand side denotes the It\^o integral of the
stochastic process $\xi_t$.\\

The aim of the present paper is the investigation of the limiting
behavior of the sequence
\begin{eqnarray}
X^{\diamond n}:=X\diamond\cdot\cdot\cdot\diamond X,\quad n\geq 1,
\end{eqnarray}
as $n$ goes to infinity. The motivation for doing this is twofold:
on one hand Wick powers of the type (1.3) appear in the formulation
of the stochastic quantization equation (see \cite{DD}); on the
other hand the Wick product, as suggested in \cite{KSS}, can be
viewed as a convolution between (generalized) random variables.
Therefore a theorem about the limiting behavior of the sequence in
(1.3) constitutes a result in the spirit of the central
limit theorem for the convolution $\diamond$.\\
We will prove that for any square integrable $X$ a properly rescaled
version of the sequence in (1.3) converges in the strong topology of
$\mathcal{L}^2(\mathcal{P})$ to a so-called stochastic exponential;
this result will imply the convergence in distribution to log-normal
random variables. We will also show that under the assumption of the
positivity of $X$ the logarithm of the above
mentioned sequence converges in distribution to a normal random variable.\\

The paper is organized as follows: Section 2 recalls some classical
background information and introduce the necessary definitions. We
refer the reader to the book \cite{N} for more detailed material.
Section 3 presents the main results of the paper together with some
important corollaries concerning convergence in distribution.

\section{ Preliminaries }

Let $(\Omega,\mathcal{F},\mathcal{P})$ be the classical Wiener space
over the time interval $[0,T]$ and denote by
$B_t(\omega):=\omega(t), t\in [0,T]$ the coordinate process which is
a Brownian motion under the measure $\mathcal{P}$. Set as usual
\begin{eqnarray*}
\mathcal{L}^2(\mathcal{P}):=\Big\{X:\Omega\to\mathbb{R}\mbox{
measurable s.t. }
E[|X|^2]:=\int_{\Omega}|X(\omega)|^2d\mathcal{P}(\omega)<+\infty\Big\},
\end{eqnarray*}
and
\begin{eqnarray*}
\Vert X\Vert:=(E[|X|^2])^{\frac{1}{2}}.
\end{eqnarray*}
According to the Wiener-It\^o chaos representation theorem any
$X\in\mathcal{L}^2(\mathcal{P})$ can be written uniquely as
\begin{eqnarray*}
X=\sum_{n\geq 0}I_n(h_n),
\end{eqnarray*}
where $I_0(h_0)=E[X]$ and for $n\geq 1$,
$h_n\in\mathcal{L}^2([0,T]^n)$ is a symmetric deterministic function
which is called the \emph{n-th order kernel of X}. Moreover for
$n\geq 1$, $I_n(h_n)$ stands for the $n$-th order multiple It\^o
integral of $h_n$ w.r.t. the Brownian motion $\{B_t\}_{0\leq t\leq
T}$. \\
By means of this representation the
$\mathcal{L}^2(\mathcal{P})$-norm of $X$ takes the following form:
\begin{eqnarray*}
\Vert X\Vert^2=\sum_{n\geq 0}n!|h_n|_{\mathcal{L}^2([0,T]^n)}^2.
\end{eqnarray*}

Given two square integrable random variables $X$ and $Y$ with
chaotic representations:
\begin{eqnarray*}
X=\sum_{n\geq 0}I_n(h_n)\mbox{ and }Y=\sum_{n\geq 0}I_n(g_n),
\end{eqnarray*}
we call \emph{Wick product of X and Y} the following quantity:
\begin{eqnarray*}
X\diamond Y:=\sum_{n\geq 0}I_n(k_n),\mbox{ with
}k_n:=\sum_{j=0}^nh_j\hat{\otimes}g_{n-j},
\end{eqnarray*}
where $\hat{\otimes}$ denotes the symmetric tensor product. We also
denote
\begin{eqnarray*}
X^{\diamond n}:=X\diamond\cdot\cdot\cdot\diamond X\mbox{
($n$-times). }
\end{eqnarray*}

In general $X\diamond Y$ does not belong to
$\mathcal{L}^2(\mathcal{P})$ since it may happen that
\begin{eqnarray*}
\Vert X\diamond Y\Vert^2=\sum_{n\geq
0}n!|k_n|_{\mathcal{L}^2([0,T]^n)}^2=+\infty.
\end{eqnarray*}

The next inequality is a straightforward generalization of Theorem 9
in \cite{KSS} where the Wick product of two random variables is
considered. This result, that will be of crucial importance in our
proofs, provides a sufficient condition for the
Wick product of random variables to be square integrable.\\
First we need to recall that for $\lambda\in\mathbb{R}$ we denote by
$\Gamma(\lambda)$ the following operator:
\begin{eqnarray*}
\Gamma(\lambda)X=\Gamma(\lambda)\sum_{n\geq 0}I_n(h_n):=\sum_{n\geq
0}\lambda^nI_n(h_n).
\end{eqnarray*}

\begin{theorem}
Let $X_1,X_2,...,X_n\in\mathcal{L}^2(\mathcal{P})$. Then
\begin{eqnarray*}
\Vert X_1\diamond\cdot\cdot\cdot\diamond X_n\Vert\leq
\Vert\Gamma(\sqrt{n})X_1\Vert\cdot\cdot\cdot\Vert\Gamma(\sqrt{n})X_n\Vert,
\end{eqnarray*}
or equivalently
\begin{eqnarray*}
\Big\Vert
\Gamma\Big(\frac{1}{\sqrt{n}}\Big)\Big(X_1\diamond\cdot\cdot\cdot\diamond
X_n\Big)\Big\Vert\leq \Vert X_1\Vert\cdot\cdot\cdot\Vert X_n\Vert.
\end{eqnarray*}
In particular for $X_1=\cdot\cdot\cdot=X_n=X$, we get
\begin{eqnarray*}
\Vert X^{\diamond n}\Vert\leq \Vert\Gamma(\sqrt{n})X\Vert^n,
\end{eqnarray*}
or equivalently
\begin{eqnarray*}
\Big\Vert \Gamma\Big(\frac{1}{\sqrt{n}}\Big)X^{\diamond
n}\Big\Vert\leq \Vert X\Vert^n.
\end{eqnarray*}
\end{theorem}

\begin{remark}
If $\lambda\in ]0,1]$, the operator $\Gamma(\lambda)$ can be
expressed in terms of the Ornstein-Uhlenbeck semigroup. In fact if
we write for $t\geq 0$,
\begin{eqnarray*}
(P_tX)(\omega):=\int_{\Omega}X(e^{-t}\omega+\sqrt{1-e^{-2t}}\tilde{\omega})d\mathcal{P}(\tilde{\omega}),
\end{eqnarray*}
then
\begin{eqnarray*}
\Gamma(\lambda)=\Gamma(e^{\log\lambda})=P_{-\log\lambda}.
\end{eqnarray*}
In particular
\begin{eqnarray*}
\Gamma\Big(\frac{1}{\sqrt{n}}\Big)=P_{\frac{1}{2}\log n}.
\end{eqnarray*}
\end{remark}

We conclude this section observing that for
$\lambda,\mu\in\mathbb{R}$,
\begin{eqnarray}
\Gamma(\mu)\Gamma(\lambda)=\Gamma(\mu\lambda),
\end{eqnarray}
and that for $X,Y\in\mathcal{L}^2(\mathcal{P})$,
\begin{eqnarray}
\Gamma(\lambda)(X\diamond
Y)=\Gamma(\lambda)X\diamond\Gamma(\lambda)Y.
\end{eqnarray}

\section{ Main results }
We are now ready to state one of the main results of this paper.

\begin{theorem}
Let $X\in\mathcal{L}^2(\mathcal{P})$ with $E[X]\neq 0$ and denote by
$h_1\in\mathcal{L}^2([0,T])$ the first-order kernel in the chaos
decomposition of $X$. Then
\begin{eqnarray*}
\Gamma\Big(\frac{1}{n}\Big)X^{\diamond
n}\in\mathcal{L}^2(\mathcal{P})\mbox{ for any }n\geq 1,
\end{eqnarray*}
and
\begin{eqnarray}
\lim_{n\to\infty}\frac{\Gamma(\frac{1}{n})X^{\diamond
n}}{E[X]^n}=\exp\Big\{\int_0^Th_1(s)dB_s-\frac{1}{2}\int_0^Th_1^2(s)ds\Big\},
\end{eqnarray}
where the convergence is in the strong topology of
$\mathcal{L}^2(\mathcal{P})$.
\end{theorem}

To ease the notation for $h\in\mathcal{L}^2([0,T])$ we set
\begin{eqnarray*}
\mathcal{E}(h):=\exp\Big\{\int_0^Th(s)dB_s-\frac{1}{2}\int_0^Th^2(s)ds\Big\}.
\end{eqnarray*}
The random variable $\mathcal{E}(h)$ belongs to
$\mathcal{L}^2(\mathcal{P})$ and its chaotic representation is
\begin{eqnarray*}
\mathcal{E}(h)=\sum_{n\geq 0}I_n\Big(\frac{h^{\otimes n}}{n!}\Big).
\end{eqnarray*}
From this identity one can easily derive the following properties:
\begin{eqnarray}
\Vert\mathcal{E}(h)\Vert\!\!&=&\!\!\exp\Big\{\frac{1}{2}|h|^2_{\mathcal{L}^2([0,T])}\Big\};\\
\Gamma(\lambda)\mathcal{E}(h)\!\!&=&\!\!\mathcal{E}(\lambda
h);\\
\mathcal{E}(h)\diamond\mathcal{E}(g)\!\!&=&\!\!\mathcal{E}(h+g).
\end{eqnarray}

\begin{proof}
First of all note that since $E[X]$ is a constant we can write
\begin{eqnarray*}
\frac{\Gamma(\frac{1}{n})X^{\diamond
n}}{E[X]^n}=\Gamma\Big(\frac{1}{n}\Big)\Big(\frac{X}{E[X]}\Big)^{\diamond
n};
\end{eqnarray*}
Therefore we can assume without loss of generality that $E[X]=1$ and
prove that
\begin{eqnarray}
\lim_{n\to\infty}\Gamma\Big(\frac{1}{n}\Big)X^{\diamond
n}=\mathcal{E}(h_1).
\end{eqnarray}

For any $n\geq 1$,
\begin{eqnarray*}
\Gamma\Big(\frac{1}{n}\Big)X^{\diamond
n}\in\mathcal{L}^2(\mathcal{P}).
\end{eqnarray*}
In fact according to Theorem 2.1,
\begin{eqnarray*}
\Big\Vert\Gamma\Big(\frac{1}{n}\Big)X^{\diamond
n}\Big\Vert&=&\Big\Vert\Gamma\Big(\frac{1}{\sqrt{n}}\Big)\Gamma\Big(\frac{1}{\sqrt{n}}\Big)X^{\diamond
n}\Big\Vert\\
&\leq&\Big\Vert\Gamma\Big(\frac{1}{\sqrt{n}}\Big)X^{\diamond
n}\Big\Vert\\
&\leq&\Vert X\Vert^n.
\end{eqnarray*}
Moreover for $n\geq 1$,
\begin{eqnarray*}
\mathcal{E}(h_1)&=&\mathcal{E}\Big(\underbrace{\frac{h_1}{n}+\cdot\cdot\cdot+\frac{h_1}{n}}_{n-times}\Big)\\
&=&\underbrace{\mathcal{E}\Big(\frac{h_1}{n}\Big)\diamond\cdot\cdot\cdot\diamond\mathcal{E}\Big(\frac{h_1}{n}\Big)}_{n-times}\\
&=&\mathcal{E}\Big(\frac{h_1}{n}\Big)^{\diamond n}.
\end{eqnarray*}

We have to prove that
\begin{eqnarray*}
\lim_{n\to\infty}\Big\Vert\Gamma\Big(\frac{1}{n}\Big)X^{\diamond
n}-\mathcal{E}(h_1)\Big\Vert=0.
\end{eqnarray*}
Since the Wick product is commutative, associative and distributive
with respect to the sum, the following identity holds:
\begin{eqnarray*}
Y^{\diamond n}-Z^{\diamond
n}=(Y-Z)\diamond\Big(\sum_{j=0}^{n-1}Y^{\diamond j}\diamond
Z^{\diamond (n-1-j)}\Big).
\end{eqnarray*}
Therefore from Theorem 2.1 and properties (3.2)-(3.4) we obtain
\begin{eqnarray*}
\Big\Vert\Gamma\Big(\frac{1}{n}\Big)X^{\diamond
n}-\mathcal{E}(h_1)\Big\Vert&=&\Big\Vert\Gamma\Big(\frac{1}{n}\Big)X^{\diamond
n}-\mathcal{E}\Big(\frac{h_1}{n}\Big)^{\diamond n}\Big\Vert\\
&=&\Big\Vert\Big(\Gamma\Big(\frac{1}{n}\Big)X-\mathcal{E}\Big(\frac{h_1}{n}\Big)\Big)\diamond
\Big(\sum_{j=0}^{n-1}\Gamma\Big(\frac{1}{n}\Big)X^{\diamond
j}\diamond \mathcal{E}\Big(\frac{h_1}{n}\Big)^{\diamond
(n-1-j)}\Big)\Big\Vert\\
&\leq&\Big\Vert\Gamma(\sqrt{2})\Big(\Gamma\Big(\frac{1}{n}\Big)X-\mathcal{E}\Big(\frac{h_1}{n}\Big)\Big)\Big\Vert\\
&&\times\Big\Vert\Gamma(\sqrt{2})\Big(\sum_{j=0}^{n-1}\Gamma\Big(\frac{1}{n}\Big)X^{\diamond
j}\diamond \mathcal{E}\Big(\frac{h_1}{n}\Big)^{\diamond
(n-1-j)}\Big)\Big\Vert\\
&=&\Big\Vert\Gamma\Big(\frac{\sqrt{2}}{n}\Big)X-\mathcal{E}\Big(\frac{\sqrt{2}h_1}{n}\Big)\Big\Vert\\
&&\times\Big\Vert\sum_{j=0}^{n-1}\Gamma\Big(\frac{\sqrt{2}}{n}\Big)X^{\diamond
j}\diamond \mathcal{E}\Big(\frac{\sqrt{2}h_1}{n}\Big)^{\diamond
(n-1-j)}\Big\Vert\\
&\leq&\Big\Vert\Gamma\Big(\frac{\sqrt{2}}{n}\Big)X-\mathcal{E}\Big(\frac{\sqrt{2}h_1}{n}\Big)\Big\Vert\\
&&\times\sum_{j=0}^{n-1}\Big\Vert\Gamma\Big(\frac{\sqrt{2}}{n}\Big)X^{\diamond
j}\diamond \mathcal{E}\Big(\frac{\sqrt{2}h_1}{n}\Big)^{\diamond
(n-1-j)}\Big\Vert\\
&\leq&\Big\Vert\Gamma\Big(\frac{\sqrt{2}}{n}\Big)X-\mathcal{E}\Big(\frac{\sqrt{2}h_1}{n}\Big)\Big\Vert\\
&&\times\sum_{j=0}^{n-1}\Big\Vert\Gamma(\sqrt{n-1})\Gamma\Big(\frac{\sqrt{2}}{n}\Big)X\Big\Vert^j
\Big\Vert\Gamma(\sqrt{n-1})\mathcal{E}\Big(\frac{\sqrt{2}h_1}{n}\Big)\Big\Vert^{n-1-j}\\
&=&\Big\Vert\Gamma\Big(\frac{\sqrt{2}}{n}\Big)X-\mathcal{E}\Big(\frac{\sqrt{2}h_1}{n}\Big)\Big\Vert\\
&&\times\sum_{j=0}^{n-1}\Big\Vert\Gamma\Big(\frac{\sqrt{2(n-1)}}{n}\Big)X\Big\Vert^j
\Big\Vert\mathcal{E}\Big(\frac{\sqrt{2(n-1)}h_1}{n}\Big)\Big\Vert^{n-1-j}.
\end{eqnarray*}
For $0\leq j\leq n-1$,
\begin{eqnarray*}
\Big\Vert\mathcal{E}\Big(\frac{\sqrt{2(n-1)}h_1}{n}\Big)\Big\Vert^{n-1-j}&=&\exp\Big\{\frac{n-1}{n^2}(n-1-j)|h_1|^2\Big\}\\
&\leq&\exp\Big\{\frac{(n-1)^2}{n^2}|h_1|^2\Big\}\\
&\leq&\exp\{|h_1|^2\}.
\end{eqnarray*}
Therefore
\begin{eqnarray*}
\Big\Vert\Gamma\Big(\frac{1}{n}\Big)X^{\diamond
n}-\mathcal{E}(h_1)\Big\Vert &\leq&
e^{|h_1|^2}\Big\Vert\Gamma\Big(\frac{\sqrt{2}}{n}\Big)X-\mathcal{E}\Big(\frac{\sqrt{2}h_1}{n}\Big)\Big\Vert\cdot
\sum_{j=0}^{n-1}\Big\Vert\Gamma\Big(\frac{\sqrt{2(n-1)}}{n}\Big)X\Big\Vert^j\\
&=&e^{|h_1|^2}\Big\Vert\Gamma\Big(\frac{\sqrt{2}}{n}\Big)X-\mathcal{E}\Big(\frac{\sqrt{2}h_1}{n}\Big)\Big\Vert
\frac{\Big\Vert\Gamma\Big(\frac{\sqrt{2(n-1)}}{n}\Big)X\Big\Vert^n-1}{\Big\Vert\Gamma\Big(\frac{\sqrt{2(n-1)}}{n}\Big)X\Big\Vert-1}.
\end{eqnarray*}
Now,
\begin{eqnarray*}
\lim_{n\to\infty}\Big\Vert\Gamma\Big(\frac{\sqrt{2(n-1)}}{n}\Big)X\Big\Vert^n&=&
\lim_{n\to\infty}\Big(1+\frac{2(n-1)}{n^2}|h_1|^2+o\Big(\frac{1}{n}\Big)\Big)^{\frac{n}{2}}\\
&=&\lim_{n\to\infty}\Big(1+\frac{2}{n}|h_1|^2+o\Big(\frac{1}{n}\Big)\Big)^{\frac{n}{2}}\\
&=&\exp{|h_1|^2},
\end{eqnarray*}
and
\begin{eqnarray*}
\lim_{n\to\infty}\frac{\Big\Vert\Gamma\Big(\frac{\sqrt{2}}{n}\Big)X-\mathcal{E}\Big(\frac{\sqrt{2}h_1}{n}\Big)\Big\Vert}
{\Big\Vert\Gamma\Big(\frac{\sqrt{2(n-1)}}{n}\Big)X\Big\Vert-1}&=&
\lim_{n\to\infty}\frac{\Big(2!\frac{4}{n^4}|h_2-h_1^{\otimes
2}|^2+o\Big(\frac{1}{n^4}\Big)\Big)^{\frac{1}{2}}}{\Big(1+\frac{2(n-1)}{n^2}|h_1|^2+o\Big(\frac{1}{n}\Big)\Big)^{\frac{1}{2}}-1}\\
&=&0.
\end{eqnarray*}
Hence we can conclude that
\begin{eqnarray*}
\lim_{n\to\infty}\Big\Vert\Gamma\Big(\frac{1}{n}\Big)X^{\diamond
n}-\mathcal{E}(h_1)\Big\Vert=0.
\end{eqnarray*}

\end{proof}

\begin{remark}
According to the observation of Remark 2.2 the statement of Theorem
3.1 can be formulated as follows:
\begin{eqnarray*}
\lim_{n\to\infty}\frac{P_{\log n}X^{\diamond
n}}{E[X]^n}=\exp\Big\{\int_0^Th_1(s)dB_s-\frac{1}{2}\int_0^Th_1^2(s)ds\Big\}.
\end{eqnarray*}
\end{remark}

Theorem 3.1 assumes that $E[X]\neq 0$. The case of zero mean random
variables is treated in the following theorem.

\begin{theorem}
Let $X\in\mathcal{L}^2(\mathcal{P})$ with $E[X]=0$. Assume that we
can find a sequence of real numbers $\{a_n\}_{n\geq 1}$ such that
\begin{eqnarray*}
\Gamma(a_n)X^{\diamond n}\in\mathcal{L}^2(\mathcal{P})\mbox{ for any
}n\geq 1,
\end{eqnarray*}
and
\begin{eqnarray*}
\lim_{n\to\infty}\Gamma(a_n)X^{\diamond n}
\end{eqnarray*}
exists in the strong topology of $\mathcal{L}^2(\mathcal{P})$. Then
the limit must be zero.
\end{theorem}

\begin{proof}
Suppose there exist a sequence of real numbers $\{a_n\}_{n\geq 1}$
and $Z\in\mathcal{L}^2(\mathcal{P})$, $Z\neq 0$ such that
\begin{eqnarray*}
\Gamma(a_n)X^{\diamond n}\in\mathcal{L}^2(\mathcal{P})\mbox{ for any
}n\geq 1,
\end{eqnarray*}
and
\begin{eqnarray*}
\lim_{n\to\infty}\Big\Vert\Gamma(a_n)X^{\diamond n}-Z\Big\Vert=0.
\end{eqnarray*}
Define
\begin{eqnarray*}
n_0:=\min\{n\geq 0: z_n\neq 0\},
\end{eqnarray*}
where $z_0=E[X]$ and for $n\geq 1$, $z_n\in\mathcal{L}^2([0,T]^n)$
is the $n$-th order
kernel in the Wiener-It\^o chaos decomposition of $Z$.\\
Since $Z$ is the strong $\mathcal{L}^2(\mathcal{P})$-limit of the
sequence $\Gamma(a_n)X^{\diamond n}$, it is also its weak limit,
that means
\begin{eqnarray*}
\lim_{n\to\infty}E[\Gamma(a_n)X^{\diamond n}U]=E[ZU],
\end{eqnarray*}
for all $U\in\mathcal{L}^2(\mathcal{P})$.\\
Since $E[X]=0$ the first non zero term in the Wiener-It\^o chaos
decomposition of $\Gamma(a_n)X^{\diamond n}$ is at least of order
$n$ and therefore for any $n>n_0$ we have
\begin{eqnarray*}
E[\Gamma(a_n)X^{\diamond n}I_{n_0}(z_{n_0})]=0,
\end{eqnarray*}
(by the orthogonality of homogenous chaos of different orders) and
hence
\begin{eqnarray*}
\lim_{n\to\infty}E[\Gamma(a_n)X^{\diamond n}I_{n_0}(z_{n_0})]=0.
\end{eqnarray*}
On the other hand
\begin{eqnarray*}
E[ZI_{n_0}(z_{n_0})]=n_0!|z_{n_0}|^2>0,
\end{eqnarray*}
by the definition of $n_0$. This means that $Z$ is not the weak
limit of the sequence $\Gamma(a_n)X^{\diamond n}$. This
contradiction completes the proof.

\end{proof}

If $\{X_n\}_{n\geq 1}$ is a sequence of random variables, by the
symbol
\begin{eqnarray*}
X_n\Rightarrow X\mbox{ as }n\to\infty,
\end{eqnarray*}
we mean that the sequence $\{X_n\}_{n\geq 1}$ converges  in
distribution as $n$ goes to infinity to the random variable $X$.

Since convergence in $\mathcal{L}^2(\mathcal{P})$ is stronger than
convergence in distribution we have the following result.

\begin{corollary}
Let $X\in\mathcal{L}^2(\mathcal{P})$ with $E[X]\neq 0$ and denote by
$h_1\in\mathcal{L}^2([0,T])$ the first-order kernel in the chaos
decomposition of $X$. Then for any $n\geq 2$ the distribution of the
random variable
\begin{eqnarray*}
\Gamma\Big(\frac{1}{n}\Big)X^{\diamond n},
\end{eqnarray*}
is absolutely continuous with respect to the Lebesgue measure on
$\mathbb{R}$. Moreover
\begin{eqnarray*}
\frac{\Gamma\Big(\frac{1}{n}\Big)X^{\diamond n}}{E[X]^n}\Rightarrow
Y\mbox{ as }n\to\infty,
\end{eqnarray*}
where $Y$ is either a log-normal random variable, more precisely
$\ln Y$ is a normal random variable with mean
$-\frac{1}{2}|h_1|_{\mathcal{L}^2([0,T])}^2$ and variance
$|h_1|_{\mathcal{L}^2([0,T])}^2$ or $Y=1$.
\end{corollary}

\begin{proof}
Observe that
\begin{eqnarray*}
\Gamma\Big(\frac{1}{n}\Big)X^{\diamond
n}&=&\Gamma\Big(\frac{1}{\sqrt{n}}\Big)\Gamma\Big(\frac{1}{\sqrt{n}}\Big)X^{\diamond
n}\\
&=&\Gamma\Big(\frac{1}{\sqrt{n}}\Big)Z
\end{eqnarray*}
where we set $Z:=\Gamma\Big(\frac{1}{\sqrt{n}}\Big)X^{\diamond n}$.
From Theorem 2.1 we know that $Z\in\mathcal{L}^2(\mathcal{P})$ and
therefore $\Gamma\Big(\frac{1}{n}\Big)X^{\diamond n}$ can be written
as the image through the operator $\Gamma(\frac{1}{\sqrt{n}})$ of a
square integrable random variable. By Theorem 4.24 in \cite{J} this
implies the absolute continuity of the distribution of
$\Gamma\Big(\frac{1}{n}\Big)X^{\diamond
n}$.\\
The rest of the proof follows from Theorem 3.1 and the fact that if
$h_1\neq 0$ then $\int_0^Th_1(s)dB_s$ is a zero mean gaussian random
variable with variance $|h_1|_{\mathcal{L}^2([0,T])}^2$.
\end{proof}

If $X$ is assumed to be non negative then Corollary 3.4 can be
reformulated as follows.

\begin{corollary}
Let $X\in\mathcal{L}^2(\mathcal{P})$ be a non negative random
variable with $\mathcal{P}(X>0)>0$ and denote by
$h_1\in\mathcal{L}^2([0,T])$ the first-order kernel in the chaos
decomposition of $X$. Then for any $n\geq 2$,
\begin{eqnarray*}
\mathcal{P}\Big(\Gamma\Big(\frac{1}{n}\Big)X^{\diamond n}>0\Big)=1,
\end{eqnarray*}
and
\begin{eqnarray*}
\ln\Gamma\Big(\frac{1}{n}\Big)X^{\diamond n}-n\ln E[X]\Rightarrow
Z\mbox{ as }n\to\infty,
\end{eqnarray*}
where $Z$ is either a normal random variable with mean
$-\frac{1}{2}|h_1|_{\mathcal{L}^2([0,T])}^2$ and variance
$|h_1|_{\mathcal{L}^2([0,T])}^2$ or $Z=0$.
\end{corollary}

\begin{proof}
The second assertion is a straightforward consequence of Corollary
3.4 since convergence in distribution is preserved under the action
of continuous functions.\\
We have to prove that for any $n\geq 2$,
\begin{eqnarray*}
\mathcal{P}\Big(\Gamma\Big(\frac{1}{n}\Big)X^{\diamond n}>0\Big)=1.
\end{eqnarray*}
Let us write as before
\begin{eqnarray*}
\Gamma\Big(\frac{1}{n}\Big)X^{\diamond
n}=\Gamma\Big(\frac{1}{\sqrt{n}}\Big)Z,
\end{eqnarray*}
where
\begin{eqnarray*}
Z=\Gamma\Big(\frac{1}{\sqrt{n}}\Big)X^{\diamond n}.
\end{eqnarray*}
We want to prove that $Z$ is non negative; according to Theorem 4.1
in \cite{NZ} this is equivalent to prove that the function
\begin{eqnarray*}
h\in\mathcal{L}^2([0,T])\mapsto
E[Z\mathcal{E}(ih)]e^{-\frac{1}{2}|h|^2_{\mathcal{L}^2([0,T])}}\in\mathbb{C},
\end{eqnarray*}
where $i$ is the imaginary unit, is positive definite. Since
$\Gamma\Big(\frac{1}{\sqrt{n}}\Big)$ is self-adjoint in
$\mathcal{L}^2(\mathcal{P})$ and for any $h\in\mathcal{L}^2([0,T])$,
\begin{eqnarray*}
E[(X\diamond Y)\mathcal{E}(h)]=E[X\mathcal{E}(h)]E[Y\mathcal{E}(h)],
\end{eqnarray*}
we get
\begin{eqnarray*}
E[Z\mathcal{E}(ih)]e^{-\frac{1}{2}|h|^2_{\mathcal{L}^2([0,T])}}&=&
E\Big[\Gamma\Big(\frac{1}{\sqrt{n}}\Big)X^{\diamond
n}\mathcal{E}(ih)\Big]e^{-\frac{1}{2}|h|^2_{\mathcal{L}^2([0,T])}}\\
&=&E\Big[X^{\diamond
n}\mathcal{E}\Big(\frac{ih}{\sqrt{n}}\Big)\Big]e^{-\frac{1}{2}|h|^2_{\mathcal{L}^2([0,T])}}\\
&=&\Big(E\Big[X\mathcal{E}\Big(\frac{ih}{\sqrt{n}}\Big)\Big]\Big)^ne^{-\frac{1}{2}|h|^2_{\mathcal{L}^2([0,T])}}\\
&=&\Big(E\Big[X\mathcal{E}\Big(\frac{ih}{\sqrt{n}}\Big)\Big]
e^{-\frac{1}{2n}|h|^2_{\mathcal{L}^2([0,T])}}e^{\frac{1}{2n}|h|^2_{\mathcal{L}^2([0,T])}}\Big)^ne^{-\frac{1}{2}|h|^2_{\mathcal{L}^2([0,T])}}\\
&=&\Big(E\Big[X\mathcal{E}\Big(\frac{ih}{\sqrt{n}}\Big)\Big]
e^{-\frac{1}{2n}|h|^2_{\mathcal{L}^2([0,T])}}\Big)^ne^{\frac{1}{2}|h|^2_{\mathcal{L}^2([0,T])}}e^{-\frac{1}{2}|h|^2_{\mathcal{L}^2([0,T])}}\\
&=&\Big(E\Big[X\mathcal{E}\Big(\frac{ih}{\sqrt{n}}\Big)\Big]
e^{-\frac{1}{2n}|h|^2_{\mathcal{L}^2([0,T])}}\Big)^n.
\end{eqnarray*}
Since $X$ is by assumption non negative, the function
\begin{eqnarray*}
E\Big[X\mathcal{E}\Big(\frac{ih}{\sqrt{n}}\Big)\Big]
e^{-\frac{1}{2n}|h|^2_{\mathcal{L}^2([0,T])}}
\end{eqnarray*}
is positive definite as well as its $n$-th power, proving the non
negativity of $Z$. \\
The image through the operator $\Gamma\Big(\frac{1}{\sqrt{n}}\Big)$
of a non negative random variable is in virtue of Corollary 4.29 in
\cite{J} a $\mathcal{P}$-a.s. positive random variable. Since
$X=\Gamma\Big(\frac{1}{\sqrt{n}}\Big)Z$ the previous assertion
applies to $X$ and completes the proof.
\end{proof}

\end{document}